\documentclass[12pt]{article}
\usepackage[utf8]{inputenc}
\usepackage[T1]{fontenc}
\usepackage[francais,english]{babel}
\usepackage{textcomp}
\usepackage{amsmath,amssymb,amsthm,amsfonts}
\usepackage{fourier}
\usepackage{geometry}
\usepackage{graphicx}
\usepackage{xcolor}
\usepackage{array}
\usepackage{calc}
\usepackage{titlesec}
\usepackage{titletoc}
\usepackage{fancyhdr}
\usepackage{titling}
\usepackage{enumitem}
\usepackage{hyperref}
\usepackage[all]{xy}
\usepackage{bbm}%
\usepackage{braket}  
\usepackage[backend=biber,style=alphabetic,sorting=nyt]{biblatex}
\usepackage{csquotes}
\usepackage{authblk}
 
\makeatletter
\newtheorem*{rep@theorem}{\rep@title}
\newcommand{\newreptheorem}[2]{%
\newenvironment{rep#1}[1]{%
 \def\rep@title{#2 \ref{##1}}%
 \begin{rep@theorem}}%
 {\end{rep@theorem}}}
\makeatother

\theoremstyle{plain}
\newtheorem{thm}{Theorem}[section]
\newreptheorem{theorem}{Theorem}
\newtheorem{prop}[thm]{Proposition}
\newtheorem{lem}[thm]{Lemma}
\newtheorem{cor}[thm]{Corollary}

\theoremstyle{definition}
\newtheorem{defi}{Definition}[section]

\theoremstyle{remark}
\newtheorem*{rem}{Remark}

\title{Explicit formulae for one-part double Hurwitz numbers with completed $3$-cycles}
\author{Viet Anh NGUYEN \thanks{ 
Email: vnguyen@math.univ-angers.fr}}
\affil{LAREMA UMR CNRS 6093, Université d'Angers, France}

\date{\vspace{-5ex}}

\begin{document}

\maketitle

\begin{abstract}
    We prove two explicit formulae for one-part double Hurwitz numbers with completed 3-cycles. We define "combinatorial Hodge integrals" from these numbers in the spirit of the celebrated ELSV formula. The obtained results imply some explicit formulae and properties of the combinatorial Hodge integrals. 
\end{abstract}

\tableofcontents

\section{Introduction}
\subsection{Notation}
First of all, let us set up some notations. For a natural number $d$, a partition $\mu$ of $d$, denoted by $\mu\vdash d$, is a sequence of non-decreasing positive integers $\mu=(\mu_1\geq \mu_2\geq\ldots\geq \mu_m)$ such that $d=\mu_1+\ldots+\mu_m$. We call $m$ the length of $\mu$, often denoted by $l(\mu)$. Another way of writing $\mu$ is $\mu= (1^{m_1},2^{m_2},\ldots)$ which tells us that $i$ appear $m_i$ times. Define $|\text{Aut}\mu|:= m_1!m_2!\ldots$. If $m_i=0$, we can choose to write $i$ or not depending on which is convenient in each context. We reserve the Greek letters for partitions. The set of all partitions of $d$ is denoted by \textbf{Part}($d$), and the set of all partitions is denoted by \textbf{Part}.         

For a series $f(z)=\sum a_iz^i$, we denote $\left[z^i\right]f:=a_i$.

\subsection{Hurwitz numbers}
Hurwitz numbers (in many variants) appeared at the cross-road of many active directions in contemporary mathematics and mathematical physics, such as the combinatorics of symmetric groups and graphs on surfaces, the intersection theory in algebraic geometry, tau functions in integrable systems and tropical geometry (see, for instance \cite{lando2013graphs,Goulden200543,kazarian2015combinatorial,cavalieri2010tropical}). 

In their simplest version, given two positive integers $d$ and $g$, and a partition $\beta$ of $d$, connected (disconnected) single Hurwitz numbers count the number of branched $d$-coverings of the Riemann sphere by a connected (possibly disconnected) Riemann surface of genus $g$, assuming that one branch point has branching given by $\beta$ and all other branch points are simple. A simple branch point is one that has branching given by the partition $\left(1^{d-2},2\right)$ of $d$. Some very interesting results connect these objects with other areas of mathematics. In particular, single Hurwitz numbers satisfy KP equations \cite{okounkov2000toda, pandharipande2000toda}, Virasoro constraints \cite{alexandrov2014enumerative} and Chekhov-Eynard-Orantin topological recursion \cite{eynard2011laplace}.

More interesting results have been proved and conjectured for other types of Hurwitz numbers. Without pretending to be exhaustive, we cite two generalizations which will be important in the sequel, together with a description of their geometric meaning:
\begin{itemize}
	\item Disconnected double Hurwitz numbers $H^{g}_{\alpha,\beta}$  are defined as the obvious generalization of the single ones (we do not consider connected numbers in this article). They count the number of branched $d$-coverings of the sphere by a possibly disconnected Riemann suface of genus $g$, where two distinguished points (say $\infty$ and $0$) have branching given by two partitions $\alpha$ and $\beta$ of $d$ while all other branch points are simple. 
	\item Disconnected double Hurwitz numbers with completed cycles, denoted by $H^{g,(r)}_{\alpha,\beta}$, are similar to standard double Hurwitz numbers. Intuitively, they count the number of branched $d$-coverings of the sphere by a possibly disconnected Riemann suface of genus $g$, where two points have branching described by $\alpha$ and $\beta$, while all other branch points have branching described by completed $(r+1)$-cycles. In particular, when $r = 1$, we recover the previous definition. A particular case of these numbers occurs when one of the partitions, say $\alpha$, is of length one, i.e. $\alpha = (d)$. In this case, we call the numbers $H^{g,(r)}_{(d),\beta}$ one-part double Hurwitz numbers with completed $(r+1)$-cycles.
\end{itemize}

Actually, enumerating branched coverings is just one of the possible sources of Hurwitz numbers. It is known that, equivalently, Hurwitz numbers enumerate permutations with some given properties, ribbon graphs (e.g. graphs on Riemann surfaces) and tropical graphs (see again \cite{lando2013graphs,Goulden200543,kazarian2015combinatorial,cavalieri2010tropical} and references therein). In particular, the interpretation in terms of permutations leads to the combinatorial definition of Hurwitz numbers in terms of irreducible characters of the symmetric groups. 

In this paper, as it is done in \cite{shadrin2012double}, we study Hurwitz numbers starting from a purely combinatorial definition. Our results are about one part double Hurwitz numbers with completed $3$-cycles, and can be seen as a natural analogue of some very well known results by Goulden, Jackson and Vakil about standard double Hurwitz numbers. 

More specifically, let $d\in\mathbb{N}$ and $\alpha,\beta$ be two partitions of $d$. The disconnected double Hurwitz number $H^g_{\alpha,\beta}$ is defined as
\begin{equation}
    H^g_{\alpha,\beta} := \frac{|\text{Aut}\alpha||\text{Aut}\beta|}{d!}\sum\frac{1}{|\text{Aut(Cov)}|},
\end{equation}
where the sum is over all degree $d$ branched covers of $\mathbb{CP}^1$ by a genus $g$ possibly disconnected Riemann surface, with $k+2$ branch points, of which $k:=k^{g}_{\alpha,\beta}=2g-2+l(\alpha)+l(\beta)$ are simple, and two have branching given by $\alpha$ and $\beta$. The condition on $k$ is a consequence of the Riemann-Hurwitz theorem. $\text{Aut(Cov)}$ means the finite automorphism group of the cover. We can define the connected double Hurwitz numbers $\tilde{H}^g_{\alpha,\beta}$ by adding the requirement that the covering surface is connected. If $\alpha=(d)$ then the connected and disconnected numbers are equal, in other words, there is no disconnected cover. The full description of the relation between connected and disconnected numbers is contained in the chamber behavior and wall-crossing formula for Hurwitz numbers \cite{johnson2015double}. 

Goulden, Jackson and Vakil computed explicitly one-part double Hurwitz numbers (in this case, as we've just said, connected numbers are equal to disconnected ones) through the following remarkably simple formula: 
\begin{thm}\cite{Goulden200543} Let $g\geq 0$ and $\beta\vdash d$ a partition of $d$. Then 
\begin{align}
    H^g_{(d),\beta}&=k!d^{k-1}\left[z^{2g}\right]\prod_{j\geq 1}\left(\frac{\sinh(jz/2)}{jz/2}\right)^{c_j}\\
    &=\frac{k!d^{k-1}}{2^{2g}}\sum_{\lambda\vdash g}\frac{\xi_{2\lambda}S_{2\lambda}}{|\text{Aut}\lambda|}.
\end{align}
\end{thm}
We postpone the definition of $c_i$, $\xi_{2\lambda}$ and $S_{2\lambda}$ to the next  section. The analogous formula for one-part double Hurwitz numbers (again, connected = disconnected in this case) with completed $3$-cycles, proved in this article, is as follows:  
\begin{thm} \label{mainthm} Given $g\geq 0$, $d>0$, let $\beta$ be a partition of odd length of $d$ and $s$ be an integer such that $2s=2g-1+l(\beta)$. Then we have:
\begin{align}
H^{g,(2)}_{(d),\beta} &= \frac{s!d^{s-1}}{2^s} \sum_{h=0}^g \frac{(2s-2h)!}{h!(s-h)!12^h}d^{2h}\left[z^{2(g-h)}\right]\prod_{i\geq 1}\left(\frac{\sinh(iz/2)}{iz/2}\right)^{c_i}  \label{eqn1}\\
&=\frac{s!d^{s-1}}{2^{s+2g}} \sum_{h=0}^g \frac{(2s-2h)!}{h!(s-h)!3^h}d^{2h} \sum_{\lambda\vdash (g-h)}\frac{\xi_{2\lambda}S_{2\lambda}}{|\text{Aut}\lambda|}.   \label{eqn2}
\end{align}
\end{thm}
Let us mention that Hurwitz numbers with completed cycles were shown to have surprising connection with relative Gromov-Witten invariants \cite{okounkov2006gromov}. It is also conjectured that there should be an ELSV-like formula for one-part double Hurwitz numbers with completed cycles \cite{shadrin2012double}. From this point of view, we believe that it is of a great use to have an efficient formula for computing these numbers. 

To conclude, let us summarize some interesting results which are direct consequences of the explicit formulae above. First, as it is in the case of standard double Hurwitz numbers, $H^{g,(2)}_{(d),\beta}$ is polynomial in the parts of $\beta$, with the highest and lowest degrees being respectively $3g + \frac{n-3}2$ and $g + \frac{n-3}2$, where $n=l(\beta)$. Note that the so-called strong piece-wise polynomiality is proven in \cite{shadrin2012double} for all double Hurwitz numbers with completed cycles. For one-part numbers, piece-wise polynomiality becomes polynomiality. Thus our formula should be viewed as an illustration of this fact through an explicitly computable case.

Second, single Hurwitz numbers are connected, through the celebrated ELSV formula, to Hodge integrals on the Deligne--Mumford moduli space of stable curves with marked points \cite{ekedahl2001hurwitz}. A similar connection between double Hurwitz numbers and integrals of cohomological classes over a (yet to be defined) moduli space is conjectured in \cite{Goulden200543} (see also \cite{shadrin2012double} for a similar conjecture for double Hurwitz numbers with completed cycles). Following what is done in \cite{Goulden200543} to support their conjecture, we defined a sort of ``combinatorial Hodge integrals'' (equation \eqref{comHodge}). Then we prove that the lowest degree Hodge integrals satisfy a formula (Proposition \ref{lambdag}) which is an analogue of the $\lambda_g$-theorem by Faber and Pandharipande \cite{faber2003hodge}. It is also an analogue of \cite[Proposition 3.12]{Goulden200543}. 

Moreover, these lowest degree combinatorial Hodge integrals, satisfy a (modified) version of the string and dilaton equations. Assembling them in a generating function $F$, we prove that the string and dilaton equations correspond to two linear operators $L_{-1}$ and $L_0$ annihilating $F$ and satisfying a Virasoro-like relation. It would be of great interest to prove that a whole set of Virasoro-like constraints can be obtained. 

Finally, we prove some closed formula for combinatorial Witten terms, i.e. combinatorial Hodge integrals of top degree.

The paper is organized as follows. Section \ref{Definition} is devoted to the combinatorial definition of double Hurwitz numbers with completed cycles. Afterwards, in section \ref{MainTheorem}, we state and prove our main result, theorem \ref{mainthm}. The corollaries coming from the main theorem are contained in the last section. 

\section*{Acknowledgements}
This article is a part of my PhD thesis that I am preparing under the supervision of Mattia Cafasso and Vladimir Roubtsov at LAREMA, UMR CNRS 6093. I am grateful to my supervisors for continuous supports. I am also grateful to Bertrand Eynard and Vincent Rivasseau for having guided my first steps in research with extreme care during my Master internship. This experience played a key role in my decision to continue the hard path of research. My PhD study is funded by the French ministerial scholarship "Allocations Spécifiques Polytechniciens".  My research is also partially supported by LAREMA and the Nouvelle Équipe "Topologie algébrique et Physique Mathématique" of the Pays de la Loire region.  

\section{Double Hurwitz numbers with completed cycles} \label{Definition}
We follow closely the exposition in \cite{shadrin2012double} that gives a combinatorial defition of double Hurwitz numbers with completed cycles.
\subsection{Shifted symmetric functions}
Let $\mathbb{Q}[x_1,\ldots,x_d]$ be the algebra of $d$-variable polynomials over $\mathbb{Q}$. The shifted action of the symmetric group $S_d$ (the group of permutations on $\{1,\ldots,d\}$) on this algebra is defined by:
\begin{equation}
\sigma(f(x_1-1,\ldots,x_d-d)):= f(x_{\sigma(1)}-\sigma(1),\ldots, x_{\sigma(d)}-\sigma(d))
\end{equation}
for $\sigma\in S_d$ and for any polynomial written in the variables $x_i-i$. Denote by $\mathbb{Q}[x_1,\ldots,x_d]^\star$ the sub-algebra of polynomials which are invariant with respect to this action.

Define the algebra of shifted symmetric functions as the projective limit  \begin{equation*}
 \Lambda^\star:=\underset{\longleftarrow}{\lim}\,\mathbb{Q}[x_1,\ldots,x_d]^\star ,  
\end{equation*}
where the projective limit is taken in the category of filtered algebras with respect to the homomorphism which sends the last variable to $0$. Concretely, an element of this algebra is a sequence $f=\{f^{(d)}\}_{d\geq 1}$, $f^{(d)}\in \mathbb{Q}[x_1,\ldots,x_d]^\star$ such that the polynomials $f^{(d)}$ are of uniformly bounded degree and stable under the restriction $f^{(d+1)}|_{x_{d+1}=0}=f^{(d)}$.
\subsection{Bases of the algebra of shifted symmetric functions}  \label{basesshifted}
\begin{defi}
For any positive integer $k$, define the corresponding shifted symmetric power sum:
\begin{equation}
p_k(x_1,x_2,\ldots):=\sum_{i=1}^\infty \left(\left(x_i-i+\frac{1}{2}\right)^k -\left(-i+\frac{1}{2}\right)^k\right).
\end{equation}
\end{defi}
In the following, we are only interested in evaluating these functions on partitions. That is, for a partition $\lambda= (\lambda_1\geq \lambda_2\geq\ldots.)$, we define $p_k(\lambda):=p_k(\lambda_1,\lambda_2,\ldots)$. As usual in symmetric function theory, for any partition $\mu$, define $p_\mu=p_{\mu_1}p_{\mu_2}\ldots$     

The functions $\{p_\mu, \mu\in\textbf{Part}\}$ form a basis of $\Lambda^\star$. Another basis is defined as follows. For partitions $\lambda$ and $\mu$ of $d$, let $\chi^\lambda_\mu$ be the irreducible character of $S_d$ associated with $\lambda$ evaluated at $\mu$ , $\dim\lambda=\chi^\lambda_{(1^n)}$ be the dimension of the irreducible representation, and $\text{Per}(\mu)$ be the set of permutations of $S_n$ whose cycle structure is described by $\mu$.
\begin{defi}
For two partitions $\lambda$ and $\mu$ of $d$, define
\begin{equation}
f_\mu(\lambda):=|\text{Per}(\mu)|\frac{\chi^\lambda_\mu}{\dim(\lambda)}.
\end{equation} 
\end{defi}
Kerov and Olshanski \cite{kerov1994polynomial} proved that:
\begin{prop}
The functions $\{f_\mu,\mu\in\textbf{Part}\}$ are shifted symmetric functions, and form a basis of $\Lambda^\star$.
\end{prop}

\subsection{Completed cycles}
Let $\mathbb{Q}S_d$ be the group algebra of $S_d$ over $\mathbb{Q}$, i.e. the algebra of formal linear sum with rational coefficients of elements of $S_d$. Let $Z\mathbb{Q}S_d$ be the center of this algebra, which is the subalgebra containing elements that commute with every element of $\mathbb{Q}S_d$. Finally, define $$Z:= \displaystyle{\bigoplus_{d=0}^{\infty}}Z\mathbb{Q}S_d.$$ 
It is well known that a basis of Z can be constructed as follows. For a partition $\mu$, let $$C_\mu:= \displaystyle{\sum_{g\in \text{Per}(\mu)}}g.$$  
Then $\{C_\mu,\mu\in\textbf{Part}\}$ form a basis of $Z$. Therefore we have the linear isomorphism 
\begin{align}
\phi:Z &\to\Lambda^\star \nonumber\\
C_\mu &\mapsto f_\mu  .  
\end{align} 
\begin{defi}
For any partition $\mu$, the completed $\mu$-conjugacy class $\overline{C}_\mu$ is defined as  $\overline{C}_\mu:=\phi^{-1}(p_\mu)/\prod\mu_i!$. Of special interest are the completed cycles $\overline{(r)}:=\overline{C}_{(r)}, r=1,2,\ldots$, $(r)$ is the $1$-part partition of $r$.
\end{defi}
Some first completed cycles are:
\begin{align*}
    0!\overline{(1)}&=(1)\\
    1!\overline{(2)}&=(2)\\
    2!\overline{(3)}&=(3)+(1,1)+\frac{1}{12}(1)\\
    3!\overline{(4)}&=(4)+2(2,1)+\frac{5}{4}(2)
\end{align*}

\subsection{Double Hurwitz numbers with completed cycles}
From now on, we shall follow the notation of Goulden, Jackson and Vakil \cite{Goulden200543}, so that the readers can compare our result with theirs easily. Let $\alpha$ and $\beta$ be two partitions of a same number $d$, whose lengths are $m$ and $n$ respectively. Let $g,r$ and $s$ be three non-negative integers such that $rs= 2g-2+m+n$. 
\begin{defi} \label{definitionnumber}
Disconnected double Hurwitz numbers with completed $(r+1)$-cycles are defined as:
\begin{equation}
H^{g,(r)}_{\alpha,\beta}:=\frac{1}{\prod\alpha_i\prod\beta_j}\sum_{\lambda\vdash d}\chi^\lambda_\alpha\left(\frac{p_{r+1}(\lambda)}{(r+1)!}\right)^s \chi^\lambda_\beta.
\end{equation}
\end{defi} 
 We often omit the superscript $(r)$ if it is fixed in advance. Since the completed $2$-cycle is equal to the ordinary $2$-cycle, the double Hurwitz numbers $H^{g,(1)}_{\alpha,\beta}$ are just the ordinary double Hurwitz numbers. 

We are mostly interested in the dependence of $H^{g,(r)}_{\alpha,\beta}$ on (the parts of) $\alpha$ and $\beta$, given fixed $g,r,d, l(\alpha)$ and $l(\beta)$. The defining formula does not give any immediate information on this because the number of terms is great (the number of partitions of $d$ increases essentially as $e^{\sqrt{d}}$), and the irreducible characters of large symmetric groups are not easy to compute. 

The numbers obtained in the case $m=1$, i.e. $\alpha=(d)$  are called one-part double numbers. In this case, the sum is simplified a lot, and we can get an explicit and compact formula for $r=2$.  

\subsection{Geometric and combinatorial interpretation}
In \cite[Section 2.2]{shadrin2015equivalence}, the authors give a combinatorial interpretation of single Hurwitz numbers with completed cycles. We can naturally generalize their construction for double numbers as follows. Let $\alpha$ and $\beta$ be two partitions of $d$. A $(\mathit{g,r,\alpha,\beta})$-\textit{factorization} $\textbf{fac}(g,r,\alpha,\beta)$ is a factorization in $S_d$ of the following form:
 \begin{equation}
     h_1\ldots h_s g_1 g_2=1,
 \end{equation}
 where $rs=2g-2+l(\alpha)+l(\beta)$, $g\in\mathbb{Z}_+$, $g_1\in\text{Per}(\alpha)$, $g_2\in\text{Per}(\beta)$, and each $h_i \in S_d$ appears in $\overline{(r+1)}$ with a coefficient $c_i\neq 0$. The weight of this factorization is defined as $$w(\textbf{fac}):=\displaystyle{\prod_{i=1}^s}c_i.$$  
 \begin{prop} We have the following equality:
     \begin{align} \displaystyle{\sum_{\textbf{fac } \in\{(\mathit{g,r,\alpha,\beta})-\textit{factorizations}\}}}w(\textbf{fac})= \frac{d!}{|\text{Aut}\alpha||\text{Aut}\beta|}H^{g,(r)}_{\alpha,\beta}.
     \end{align} 
 \end{prop}
 \begin{proof}
 Since $\{C_\lambda|\lambda\vdash d\}$ form a basis of $Z\mathbb{Q}S_d$, we can write: 
 \begin{equation} \label{classexpansion}
  \overline{(r+1)}^s C_\alpha C_\beta = \sum_{\lambda\vdash d}a_\lambda C_\lambda   .
 \end{equation} 
 By definition:
 \begin{align*}
     \displaystyle{\sum_{\textbf{fac } \in\{(\mathit{g,r,\alpha,\beta})-\textit{factorizations}\}}}w(\textbf{fac})= \left[C_{(1^d)}\right]\underbrace{\overline{(r+1)}\ldots \overline{(r+1)}}_{s}C_\alpha C_\beta = a_{(1^d)} ,
 \end{align*}
 where the right hand side means the coefficient of $C_{(1^d)}=\text{Id}\equiv 1$ in the product following it. Consider the left regular representation of $\mathbb{Q}S_d$, i.e. the action of $\mathbb{Q}S_d$ on itself by multiplication on the left. A main theorem of the representation theory of the symmetric groups\cite[Theorem A.1.5]{lando2013graphs} gives us the decomposition of this representation into irreducible ones: 
 \begin{align*}
     \mathbb{Q}S_d = \bigoplus_{\lambda\vdash d}\dim(\lambda)V_\lambda.
 \end{align*}
 Here $\dim(\lambda)$ is the dimension of the irreducible representation $\lambda$ of $S_d$ (as we defined in the section \ref{basesshifted}) and $\dim V_\lambda=\dim(\lambda)$.  The action of an element $B\in Z\mathbb{Q}S_d$ in $V_\lambda$ is multiplication by a number $L_\lambda(B)$, i.e. the matrix $L(B)$ representing $B$ is diagonal:
 \begin{align*}
     L(B)= \displaystyle{\text{diag}_{\lambda\vdash d}}\left(\underbrace{L_\lambda(B)}_{\dim(\lambda)^2 \text{times}}\right).
 \end{align*}
 In particular, we can compute:
 \begin{align*}
     L_\lambda(C_\alpha)&=f_\alpha(\lambda)= |\text{Per}(\alpha)|\frac{\chi^\lambda_\alpha}{\dim(\lambda)},\\
     L_\lambda(\overline{(r+1)})&= \frac{1}{(r+1)!} p_{r+1}(\lambda)= \frac{1}{(r+1)!} \sum_{i=1}^{l(\lambda)} \left(\left(\lambda_i-i+\frac{1}{2}\right)^{r+1} -\left(-i+\frac{1}{2}\right)^{r+1}\right).
 \end{align*}
Now let us take the trace of the action in the left regular representation of the two sides of the equation \eqref{classexpansion}. The right hand side gives $d! a_{(1^d)}$ since \begin{equation*}
    \textbf{Tr}L(g)= \begin{cases} & d! \quad\text{ if } g=1, \\
                                   & 0 \quad\text{ otherwise} \end{cases}    
\end{equation*} 
while the left hand side gives 
\begin{align*}
 \textbf{Tr}\left(L(\overline{(r+1)})^s
     L(C_\alpha)L(C_\beta)\right) = \sum_{\lambda\vdash d}\dim(\lambda)^2 L_\lambda(C_\alpha)L_\lambda(C_\beta)L_\lambda(\overline{(r+1)})^s .  
\end{align*} 
Finally, we get:
 \begin{align*}
     \left[C_{(1^d)}\right]\underbrace{\overline{(r+1)}\ldots \overline{(r+1)}}_{s}C_\alpha C_\beta &=
     \frac{1}{d!}\sum_{\lambda\vdash d}\dim(\lambda)^2 L_\lambda(C_\alpha)L_\lambda(C_\beta)L_\lambda(\overline{(r+1)})^s\\
     &= \frac{|\text{Per}(\alpha)||\text{Per}(\beta)|}{d!}\sum_{\lambda\vdash d} \chi^\lambda_\alpha \chi^\lambda_\beta \left(\frac{p_{r+1}(\lambda)}{(r+1)!}\right)^s \\
     &= \frac{d!}{|\text{Aut}\alpha||\text{Aut}\beta|}H^{g,(r)}_{\alpha,\beta}.
 \end{align*}
 In the last line, we used: 
 \[ |\text{Per}(\alpha)|=\frac{d!}{|\text{Aut}\alpha|\prod\alpha_i}.\]
 \end{proof}
 
 Double Hurwitz numbers with completed cycles also have a beautiful geometric interpretation. $H^{g,(r)}_{\alpha,\beta}$ intuitively counts the number of "weighted" branched $d$-coverings of the sphere by a possibly disconnected Riemann suface of genus $g$, where two points have branching described by $\alpha$ and $\beta$, while all other branch points have branching described by completed $(r+1)$-cycles. In fact, the precise geometric picture is more delicate but not necessary for us in the sequel. We refer to \cite[Section 2.5]{shadrin2012double} for details. 
 
Instead, still on the intuitive level, we explain what is happening to help the readers who are not specialists. We invite the readers to the excellent book \cite{lando2013graphs}, in particular Appendix A, where the following things are discussed in details. Let $\mu^1,\ldots,\mu^j$ be $j$ partitions of a natural number $d$. Denote by $N^g_{\mu^1,\ldots,\mu^j}$ the weighted number of not necessarily connected degree $d$ branched covers of $\mathbb{CP}^1$ by a Riemann surface of genus $g$ which have exactly $j$ fixed ramification points, whose profiles are given by $\mu^1,\ldots,\mu^j$. The weight of a cover is the inverse of the order of its finite automorphism group. Note that $g$ is determined by $\mu^1,\ldots,\mu^j$ (and $d$) according to the Riemann-Hurwitz formula. By an analysis of the monodromy of the covers, we can interpret this number in terms of factorizations:
 \begin{equation*}
     N^g_{\mu^1,\ldots,\mu^j}= \#\big\{(g_1,\ldots,g_j)\in \text{Per}(\mu^1)\times\ldots\text{Per}(\mu^j)| g_1\ldots g_j=1\big\}.
 \end{equation*}
 Using arguments as in the above proof, Frobenius proved that:
 \begin{equation*}
    N^g_{\mu^1,\ldots,\mu^j}= \frac{|\text{Per}(\mu^1)|\ldots|\text{Per}(\mu^j)|}{d!} \sum_{\lambda\vdash d}\frac{\chi^\lambda_{\mu^1}\ldots\chi^\lambda_{\mu^j}}{\dim(\lambda)^{j-2}}.
 \end{equation*}
In particular, the disconnected ordinary double Hurwitz numbers are given by:
\begin{equation*}
    H^g_{\alpha,\beta}= \frac{1}{\prod\alpha_i\prod\beta_i} \sum_{\lambda\vdash d}\chi^\lambda_\alpha\chi^\lambda_\beta f^{k}_{(2,1^{d-2})}(\lambda),
\end{equation*}
where $k$ is the number of simple branch points. Therefore, if we want to impose the condition that the branching over points different from $0$ and $\infty$ is given by completed cycles, we need to replace the shifted symmetric function $f_\mu$ by the shifted symmetric power sum $p_\mu$, and hence have the definition \ref{definitionnumber}. 

\section{Explicit formula for one-part double Hurwitz numbers with completed 3-cycles}\label{MainTheorem}
We consider the case $r=2$. Let $\beta$ be a partition of $d$ of odd length $n$. Let $s=g+ \frac{n-1}{2}$. We write $\beta$ in three ways, each of which is convenient in each specific context.
\begin{equation}
    \left(\beta_1,\beta_2,\ldots\right)= \left(1^{n_1}2^{n_2}\ldots\right)=\left(\ell^{n_\ell}\ldots q^{n_q}\right).
\end{equation}
Here $\ell$ and $q$ are the smallest and greatest numbers appearing in $\beta$. If a number $i$ does not appear in $\beta$, $n_i=0$. Let $c_i = n_i $ for $i\geq 2$ and $c_1=n_1-1$. We have $\sum_i c_i= n-1$ and $\sum ic_i=d-1$.

We can easily compute that:
\begin{align}
p_3\left(\left(d-k,1^k\right)\right)= \left(d-k-\frac{1}{2}\right)^3 -\left(-k-\frac{1}{2}\right)^3 = 3d\left(\left(k - \frac{d-1}{2}\right)^2 +\frac{d^2}{12}\right)   . 
\end{align}
\begin{rem} The fact that $p_3\left(\left(d-k,1^k\right)\right)$ has the form $a(k+b)^2 + c$, where $a,b,c$ do not depend on $k$ (but on $d$ of course) turns out to be crucial for our method. Unfortunately, $p_{r+1}((d-k,1^k))$ for $r\geq 3$ do not have the form $a(k+b)^r +c$ , so we do not get a compact formula by the same strategy.  
\end{rem}

\begin{lem} \label{lemmaChi} We have the following irreducible character evaluation:
\begin{align} 
\chi^{(d-k,1^k)}_\beta& = (-1)^k\left[z^k\right](1+z+\ldots+z^{\ell-1})(1-z^\ell)^{n_\ell-1}\prod_{i\geq \ell+1}(1-z^i)^{n_i} \nonumber\\
                      &= (-1)^k\left[z^k\right]\prod_{i\geq 1}(1-z^i)^{c_i} \nonumber\\
                      &= (-1)^k \sum_{h=0}^{\ell-1}\sum_{j_\ell=0}^{n_\ell-1}\sum_{j_{\ell+1}=0}^{n_{\ell+1}}\ldots\sum_{j_q=0}^{n_q}(-1)^{\sum_{i\geq l}j_i}\binom{n_\ell-1}{j_\ell}\binom{n_{\ell+1}}{j_{\ell+1}}\ldots\binom{n_q}{j_q} \delta_{k,h+\sum_{i=\ell}^q{ij_i}}.
\end{align}
\end{lem}  
Here, $\delta_{x,y}:=1$ if $x=y$, and $0$ otherwise. This lemma is well known, and can be derived from the Murnaghan-Nakayama rule. See, for instance, \cite[p.59]{Goulden200543}. 

For $j\geq 1$, let $\xi_{2j}=\left[x^{2j}\right]\log(\sinh x/x)$ and 
\begin{equation*} 
S_{2j}=\sum_{k\geq 1}k^{2j}c_k= -1 + \sum_{k\geq 1}k^{2j}n_k=-1+\sum_{k\geq 1}{\beta_k}^{2j},
\end{equation*}
i.e. $S_{2j}$ is a power sum for the partition, shifted by $1$. For a partition $\lambda$ , let $\xi_\lambda=\xi_{\lambda_1}\xi_{\lambda_2}\ldots$ and $S_\lambda = S_{\lambda_1}S_{\lambda_2}\ldots$ and $2\lambda=(2\lambda_1,2\lambda_2,\ldots)$.

\begin{lem}\label{lemmaS}The following formula holds true:
\begin{equation} 
\left[z^{2k}\right]\prod_{i\geq 1}\left(\frac{\sinh(iz/2)}{iz/2}\right)^{c_i} = 2^{-2k}\sum_{\lambda\vdash k}\frac{\xi_{2\lambda}S_{2\lambda}}{|\text{Aut}\lambda|}.
\end{equation}
\end{lem} 

\begin{proof}
We have
\begin{align*}
    \prod_{i\geq 1}\left(\frac{\sinh(ix)}{ix}\right)^{c_i}= \exp\left(\sum_{i\geq 1}c_i\sum_{j\geq 1}\xi_{2j}i^{2j}x^{2j}\right)= \exp\left(\sum_{j\geq 1}\xi_{2j}S_{2j}x^{2j}\right)= \sum_{\lambda}\frac{\xi_{2\lambda}S_{2\lambda}}{|\text{Aut}\lambda|}x^{2|\lambda|}.
\end{align*}
The proof is finished upon setting $x=z/2$.
\end{proof}

\begin{lem}\label{lemmaSp}
Let $S_p(k,x) := \sum_{h=0}^{k-1}(h+x)^p$. Then 
\begin{align}
   S_p(k,x)=\left[ \frac{z^p}{p!}\right] e^{xz}\left(1+e^z+\ldots+e^{(k-1)z}\right). 
\end{align}
\end{lem}
\begin{proof}
Indeed, 
\begin{align*}
   \sum_{p=0}^{\infty}S_p(k,x)\frac{z^p}{p!} = \sum_{h=0}^{k-1}e^{z(h+x)}=e^{zx}\left(1+e^z+\ldots+e^{(k-1)z}\right) .
\end{align*}
\end{proof}

Now, we prove our main Theorem:
\begin{reptheorem}{mainthm}  Given $g\geq 0$, $d>0$, let $\beta$ be a partition of odd length of $d$ and $s$ be an integer such that $2s=2g-1+l(\beta)$. Then we have:
\begin{align}
H^{g,(2)}_{(d),\beta} &= \frac{s!d^{s-1}}{2^s} \sum_{h=0}^g \frac{(2s-2h)!}{h!(s-h)!12^h}d^{2h}\left[z^{2(g-h)}\right]\prod_{i\geq 1}\left(\frac{\sinh(iz/2)}{iz/2}\right)^{c_i} \\
&=\frac{s!d^{s-1}}{2^{s+2g}} \sum_{h=0}^g \frac{(2s-2h)!}{h!(s-h)!3^h}d^{2h} \sum_{\lambda\vdash (g-h)}\frac{\xi_{2\lambda}S_{2\lambda}}{|\text{Aut}\lambda|}. 
\end{align}
\end{reptheorem}

\begin{proof}
By definition:
\begin{equation}
H^{g,(2)}_{(d),\beta}=\frac{1}{d\prod\beta_j}\sum_{\lambda\vdash d}\chi^\lambda_{(d)}\left(\frac{p_3(\lambda)}{6}\right)^s \chi^\lambda_\beta.
\end{equation} 
It is well known that $\chi^\lambda_{(d)}=0$ except for $\lambda=(d-k,1^k), k=0,\ldots,d-1$, in which case it is equal to $(-1)^k$. 
So we have:
\begin{align*}
H^{g,(2)}_{(d),\beta}&= \frac{d^{s-1}}{2^s\prod\beta_j}\sum_{k=0}^{d-1}\left(\left(k - \frac{d-1}{2}\right)^2 +\frac{d^2}{12}\right)^s (-1)^k \chi^{(d-k,1^k)}_\beta \\
&= \frac{d^{s-1}}{2^s\prod\beta_j}\left[\frac{t^s}{s!}\right]\sum_{k=0}^{d-1}\exp\left\{t\left(\left(k - \frac{d-1}{2}\right)^2 +\frac{d^2}{12}\right)\right\} (-1)^k \chi^{(d-k,1^k)}_\beta \\
&= \frac{s!d^{s-1}}{2^s\prod\beta_j}\left[t^s\right] \exp\left(\frac{td^2}{12} \right)\sum_{k=0}^{d-1}\exp\left\{t\left(k - \frac{d-1}{2}\right)^2 \right\} (-1)^k \chi^{(d-k,1^k)}_\beta.
\end{align*}

We treat the sum separately now:
\begin{align*}
A&=\sum_{k=0}^{d-1}\exp\left\{t\left(k - \frac{d-1}{2}\right)^2 \right\} (-1)^k \chi^{(d-k,1^k)}_\beta \\
&\overset{\text{Lemma } \ref{lemmaChi}}{=} \sum_{h=0}^{\ell-1}\sum_{j_\ell=0}^{n_\ell-1}\sum_{j_{\ell+1}=0}^{n_{\ell+1}}\ldots\sum_{j_{q}=0}^{n_{q}}\exp\left\{t\left(h+\sum_{i\geq \ell}ij_i- \frac{d-1}{2}\right)^2 \right\}(-1)^{\sum_{i\geq l}j_i}\binom{n_\ell-1}{j_\ell}\binom{n_{\ell+1}}{j_{\ell+1}}\ldots\binom{n_q}{j_q}.
\end{align*}

We expand the exponential and sum over $h$ first :
\begin{align*}
A &= \sum_{j_\ell=0}^{n_l-1}\sum_{j_{\ell+1}=0}^{n_{\ell+1}}\ldots\sum_{j_{q}=0}^{n_{q}}(-1)^{\sum_{i\geq \ell}j_i}\binom{n_\ell-1}{j_\ell}\binom{n_{\ell+1}}{j_{\ell+1}}\ldots\binom{n_q}{j_q}\sum_{p=0}^{\infty}\frac{t^p}{p!} \sum_{h=0}^{\ell-1}\left(h+\sum_{i\geq \ell}ij_i- \frac{d-1}{2}\right)^{2p} \\
&\overset{\text{Lem }\ref{lemmaSp}}{=}\sum_{p=0}^{\infty}\frac{(2p)!t^p}{p!}\left[z^{2p}\right] \left(1+e^z+\ldots+e^{(\ell-1)z}\right)e^{-\frac{(d-1)z}{2}}\times\\
&\times\sum_{j_\ell=0}^{n_\ell-1}\sum_{j_{\ell+1}=0}^{n_{\ell+1}}\ldots\sum_{j_{q}=0}^{n_{q}}e^{z\sum_{i\geq \ell}ij_i}(-1)^{\sum_{i\geq \ell}j_i}\binom{n_\ell-1}{j_l}\binom{n_{\ell+1}}{j_{\ell+1}}\ldots \binom{n_q}{j_q}\\
&=\sum_{p=0}^{\infty}\frac{(2p)!t^p}{p!}\left[z^{2p}\right] e^{-\frac{(d-1)z}{2}}\left(1+e^z+\ldots+e^{(\ell-1)z}\right)\left(1-e^{\ell z}\right)^{n_\ell-1}\prod_{i\geq \ell+1}\left(1-e^{iz}\right)^{n_i} \\
&=\sum_{p=0}^{\infty}\frac{(2p)!t^p}{p!}\left[z^{2p}\right] e^{-\frac{(d-1)z}{2}}\prod_{i\geq 1}\left(1-e^{iz}\right)^{c_i}.
\end{align*}

Finally we get:
\begin{align*}
H^{g,(2)}_{(d),\beta} &= \frac{s!d^{s-1}}{2^s\prod\beta_j}\left[t^s\right] \exp\left(\frac{td^2}{12} \right)\sum_{p=0}^{\infty}\frac{(2p)!t^p}{p!}\left[z^{2p}\right] e^{-\frac{(d-1)z}{2}}\prod_{i\geq 1}\left(1-e^{iz}\right)^{c_i} \\
&=  \frac{s!d^{s-1}}{2^s}\left[t^s\right] \exp\left(\frac{td^2}{12} \right)\sum_{p=0}^{\infty}\frac{(2p)!t^p}{p!}\left[z^{2p-n+1}\right]\prod_{i\geq 1}\left(\frac{\sinh(iz/2)}{iz/2}\right)^{c_i} \\
&= \frac{s!d^{s-1}}{2^s} \sum_{h=0}^s \frac{(2s-2h)!}{h!(s-h)!12^h}d^{2h}\left[z^{2s-2h-n+1}\right]\prod_{i\geq 1}\left(\frac{\sinh(iz/2)}{iz/2}\right)^{c_i}.
\end{align*}
To pass from the first line to the second, we write $1-e^{iz}= -2 e^{iz/2}\sinh(iz/2)$ and use $\sum_i c_i = n-1$, $\sum_i ic_i=d-1$ and $\prod_i i^{c_i}=\prod\beta_j$. There is also the factor $(-1)^{\sum c_i}=(-1)^{n-1}=1$ since $n$ is odd. 

Note that $2s=2g-1+n$, so we are taking the coefficient of $z^{2(g-h)}$. Because the lowest degree of the series in $z$ is $0$, the summing index $h$ actually runs from $0$ to $g$. Finally, we get:
\begin{align*}
H^{g,(2)}_{(d),\beta} = \frac{s!d^{s-1}}{2^s} \sum_{h=0}^g \frac{(2s-2h)!}{h!(s-h)!12^h}d^{2h}\left[z^{2(g-h)}\right]\prod_{i\geq 1}\left(\frac{\sinh(iz/2)}{iz/2}\right)^{c_i}.
\end{align*}
Using lemma \ref{lemmaS}, we obtain the second equation of the theorem \ref{mainthm}.
\end{proof}
\section{Some corollaries} \label{Conse}
Our formula is explicit and computationally efficient. We observe a strong similarity with the case of ordinary one-part double Hurwitz numbers. Consequently, as in \cite{Goulden200543}, we prove some fairly important implications.
\subsection{Strong Polynomiality} \label{strpoly}
Our formula gives immediately the strong polynomialty of $1$-part double Hurwitz numbers with completed 3-cycles. In fact, double Hurwitz numbers with completed cycles of any size satisfy the strong piecewise polynomiality, i.e. they are piecewise polynomial with the highest and lowest orders respectively $(r+1)s+1-m-n$ and $(r+1)s+1-m-n-2g$. This is proved in \cite{shadrin2012double}. For one-part numbers, piecewise polynomiality becomes polynomiality. Our formula should be viewed as an illustration of this fact through an explicitly computable case. 
\begin{cor}
$H^{g,(2)}_{(d),\beta}$, for fixed $g$ and $n$, is a polynomial of the parts of $\beta$ and satisfies the strong polynomiality property, i.e. it is polynomial in $\beta_1,\beta_2,\ldots$ with highest and lowest degrees respectively $3g+\frac{n-3}{2}$ and $g+\frac{n-3}{2}$ 
\end{cor}
The polynomial is divisible by  $d^{s-1}$, but unlike the case of ordinary double Hurwitz numbers, $2^sH^{g,(2)}_{(d),\beta}/s!d^{s-1}$ depends on the number of parts of $\beta$ for $g\geq 1$. See the comment after \cite[Corollary 3.2]{Goulden200543}.

\subsection{Connection with intersection theory on moduli spaces of curves and "the $\lambda_g$ theorem"}
The connected single Hurwitz number $H^g_\beta$  is defined to be the weighted number of degree $d$ branched covers of $\mathbb{CP}^1$ by a connected Riemann surface of genus $g$, with $p+1$ branch points, of which $p$ are simple, and one has branching given by $\beta$. Due to the Riemann-Hurwitz formula, we have $p=2g-2+d+n$, where $n=l(\beta)$. The celebrated ELSV formula \cite{ekedahl2001hurwitz} connects these numbers with integrals on the moduli space of stable curves:
\begin{equation}
    H^g_\beta=C(g,\beta)\int_{\overline{\mathcal{M}}_{g,n}}\frac{1-\lambda_1+\lambda_2-\ldots+(-1)^g\lambda_g}{(1-\beta_1\psi_1)\ldots(1-\beta_n\psi_n)},
\end{equation}
where 
\begin{equation}
    C(g,\beta) = p!\prod_{i=1}^n\frac{\beta_i^{\beta_i}}{\beta_i!}.
\end{equation}
$\overline{\mathcal{M}}_{g,n}$ is the moduli space of stable curves of genus $g$ with $n$ marked points. $\lambda_i$ is a certain codimension $i$ class, and $\psi_i$ is a certain codimension $1$ class. The precise definitions are not needed for us. The reader is invited to consult the original papers or the very comprehensible book \cite{lando2013graphs}.

In other words, $P^g\left(\beta\right):=H^g_\beta/C(g,\beta)$ is polynomial in $\beta_1,\ldots,\beta_n$ and the (linear) Hodge integrals are given by:
\begin{equation} \label{Hodgeintegral}
\langle\tau_{b_1}\ldots\tau_{b_n}\lambda_k\rangle_g:=\int_{\overline{\mathcal{M}}_{g,n}}\psi_1^{b_1}\ldots\psi_n^{b_n}\lambda_k=(-1)^k\left[\beta_1^{b_1}\ldots\beta_n^{b_n}\right]P^g\left(\beta\right)  .
\end{equation}
Another ELSV formula has been found for the so-called orbifold Hurwitz numbers, i.e double Hurwitz numbers with $\alpha=(a,a,\ldots,a)$ by Johnson, Pandharipande and Tseng \cite{johnson2011abelian}. It is an important and challenging problem to find other ELSV formulas. An important clue is the (piece-wise) polynomiality of Hurwitz numbers.

In \cite{shadrin2012double}, the authors conjecture that for every $r\geq 1$, there exist moduli spaces $X_{g,n}^{(r)}$ of complex dimension $2g(r+1)+n-1$ such that we have the following ELSV formula:
\begin{equation}
    H^{g,(r)}_{(d),\beta}= \frac{s!}{d}\int_{X_{g,n}}\frac{1-\Lambda_2+\Lambda_4-\ldots+(-1)^g\Lambda_{2g}}{(1-\beta_1\Psi_1)\ldots(1-\beta_n\Psi_n)},
\end{equation}
where we fix the degrees of the rational cohomology classes $\Lambda_{2k}\in H^{4rk}\left(X_{g,n}^{(r)}\right)$ and $\Psi_i\in H^{2r}\left(X_{g,n}^{(r)}\right)$. 

A similar conjecture was previously made by Goulden, Jackson and Vakil \cite{Goulden200543} for ordinary double Hurwitz numbers, i.e. the case $r=1$. To support their conjecture, they made a thorough combinatorial study and found many similarity between "combinatorial Hodge integrals" and the "real" ones defined by \eqref{Hodgeintegral} .

Following them, let us define the \textit{combinatorial Hodge integrals} for $b_1,\ldots,b_n\geq 0$ and $0\leq k\leq g$: 
\begin{align}\label{comHodge}
\langle\langle\tau_{b_1}\ldots\tau_{b_n}\Lambda_{2k}\rangle\rangle_g:=(-1)^k\left[\beta_1^{b_1}\ldots\beta_n^{b_n}\right]\left(d\frac{H^{g,(2)}_{(d),\beta}}{s!}\right)=(-1)^k\left[\beta_1^{b_1}\ldots\beta_n^{b_n}\right]\left(d\frac{H^{g,(2)}_{(d),\beta}}{\left(g+\frac{n-1}{2}\right)!}\right).
\end{align}
We use the double brackets just to remind the readers that the conjectural moduli space has not been found, and do not keep the superscript $(2)$ to save space. This "intersection" number vanishes unless $b_1+\ldots+b_n+2k=3g+\frac{n-1}{2}$. The order of this integral is defined to be $b_1+\ldots+b_n$.

We are going to evaluate the lowest order terms, i.e. the terms with $k=g$. In \cite{faber2003hodge}, Faber and Pandharipande proved the so-called $\lambda_g$ conjecture (which is also a consequence of the unresolved Virasoro conjecture):
\begin{equation}
    \langle\tau_{b_1}\ldots\tau_{b_n}\lambda_g\rangle_g=\mathbf{c}_g\binom{2g-3+n}{b_1,\ldots,b_n}.
\end{equation}
By computing $\langle\tau^{2g-2}\lambda_g\rangle_g$, they found
\begin{equation*}
\mathbf{c}_g=\frac{2^{2g-1}-1}{2^{2g-1}(2g)!}|B_{2g}|    ,
\end{equation*}
where $B_{2g}$ is a Bernoulli number ($B_0=1, B_2=1/6, B_4=-1/30,B_6=1/42,\ldots)$. In analogy with this theorem, the following combinatorial version is proved in \cite[Proposition 3.12]{Goulden200543}:
\begin{equation}
    \langle\langle\tau_{b_1}\ldots\tau_{b_n}\Lambda_{2g}\rangle\rangle_g^{r=1}=\mathbf{c}_g\binom{2g-3+n}{b_1,\ldots,b_n}.
\end{equation}
Their symbol $\langle\langle.\rangle\rangle_g^{r=1}$ is defined just like \eqref{comHodge}, with a different normalisation \cite[Eq.25]{Goulden200543}. It is quite remarkable that the same constant $\mathbf{c}_g$ appears. However, the most interesting thing about these two formulas is the multinomial coefficient. 

Here, thank to the main Theorem \ref{mainthm}, we can also easily evaluate $\langle\langle\tau_{b_1}\ldots\tau_{b_n}\Lambda_{2g}\rangle\rangle_g$. 
\begin{prop} \label{lambdag} For $b_1+\ldots+b_n=g+\frac{n-1}{2}$, the lowest combinatorial Hodge integral is given by:
\begin{equation} 
    \langle\langle\tau_{b_1}\ldots\tau_{b_n}\Lambda_{2g}\rangle\rangle_g= \binom{g+\frac{n-1}{2}}{b_1,\ldots,b_n}\mathbf{C}_{g,n},
\end{equation}
with 
\begin{equation}
    \mathbf{C}_{g,n}=\frac{(2g+n-1)!\left(2^{2g-1}-1\right)}{(2g)!\left(g+\frac{n-1}{2}\right)!2^{3g+\frac{n-3}{2}}}|B_{2g}|.
\end{equation}
\end{prop}

\begin{proof}
First, we extract the lowest term in the polynomial $dH^{g,(2)}_{(d),\beta}/s!$ . This means simply taking only the $h=0$ term in the sum \eqref{eqn2}, and then taking the constant term of $S_{2\lambda}$, which is $(-1)^{l(\lambda)}$, in the sum over all partitions $\lambda$ of $g$. The result is: 
\begin{equation*}
    \frac{(2s)!d^s}{s!2^{s+2g}}\sum_{\lambda\vdash g}\frac{(-1)^{l(\lambda)}\xi_{2\lambda}}{|\text{Aut}\lambda|}= \frac{\left(2g+n-1\right)!d^{g+\frac{n-1}{2}}}{\left(g+\frac{n-1}{2}\right)!2^{3g+\frac{n-1}{2}}}\sum_{\lambda\vdash g}\frac{(-1)^{l(\lambda)}\xi_{2\lambda}}{|\text{Aut}\lambda|}.
\end{equation*}
Then we compute the coefficient of $\beta_1^{b_1}\ldots\beta_n^{b_n}$ to get the lowest combinatorial Hodge integral:
\begin{align} \label{Lowest1}
\langle\langle\tau_{b_1}\ldots\tau_{b_n}\Lambda_{2g}\rangle\rangle_g &= (-1)^g\binom{g+\frac{n-1}{2}}{b_1,\ldots,b_n}\frac{(2g+n-1)!}{\left(g+\frac{n-1}{2}\right)!2^{3g+\frac{n-1}{2}}}\sum_{\lambda\vdash g}\frac{(-1)^{l(\lambda)}\xi_{2\lambda}}{|\text{Aut}\lambda|}   .
\end{align}
On the other hand, using \eqref{eqn1}, we compute: 
\begin{align}\label{Lowest2}
\langle\langle\tau_g\Lambda_{2g}\rangle\rangle_g&= (-1)^g\left[d^g\right]\frac{dH_{(d),(d)}}{g!} \nonumber\\
&= (-1)^g\left[d^g\right]\frac{d^g(2g)!}{2^gg!} \left[z^{2g}\right] \frac{z/2}{\sinh z/2}\frac{\sinh dz/2}{dz/2}\nonumber\\
&=\frac{(2g)!}{g!2^{g}}(-1)^g\left[z^{2g}\right]\frac{z/2}{\sinh z/2} \nonumber\\
&=\frac{2^{2g-1}-1}{g!2^{3g-1}}|B_{2g}|.
\end{align}
Comparing \eqref{Lowest1} and \eqref{Lowest2}, we deduce $\mathbf{C}_{g,n}$ as written.
\end{proof}

We observe a strong similarity with the results quoted above. The main difference is the dependence on $n$ of the factor $\mathbf{C}_{g,n}$. A geometric explanation would be of great interest.

\subsection{Dilaton and string equations}
Goulden, Jackson and Vakil proved that their combinatorial Hodge integrals for ordinary double Hurwitz numbers satisfy the string and dilation equations  \cite[Proposition 3.10]{Goulden200543}. Here we prove that the lowest terms satisfy the (modified) string and dilaton equations for every genus $g$. The situation for higher terms are not clear to us.
\begin{prop}\textbf{String equation:} For $g\geq 0$, $n\geq 1$, $n$ odd, $b_1,\ldots,b_n\geq 0$, $b_1+\ldots +b_n=g+\frac{n+1}{2}$: 
\begin{align}
\langle\langle\tau_0^2\tau_{b_1}\ldots\tau_{b_n}\Lambda_{2g}\rangle\rangle_g = (2g+n)\sum_{i=1}^n \langle\langle \tau_{b_1}\ldots\tau_{b_{i-1}}\tau_{b_{i}-1}\tau_{b_{i+1}}\ldots\tau_{b_{n}}\Lambda_{2g}\rangle\rangle_g.\end{align}
\textbf{Dilaton equation:} For $g\geq 0$, $n\geq 1$, $n$ odd, $b_1,\ldots,b_n\geq 0$, $b_1+\ldots +b_n=g+\frac{n-1}{2}$ (minus here is not a misprint):
\begin{align}
\langle\langle\tau_0\tau_1\tau_{b_1}\ldots\tau_{b_{n}}\Lambda_{2g}\rangle\rangle_g = (2g+n)\left(g+\frac{n+1}{2}\right) \langle\langle\tau_{b_1}\ldots\tau_{b_{n}}\Lambda_{2g}\rangle\rangle_g.
\end{align}
Here we assume that $\langle\langle.\rangle\rangle$=0 if there is some $\tau_{<0}$ inside.
\end{prop}
\begin{proof}
For the string equation: 
\begin{align*}
\langle\langle\tau_0^2\tau_{b_1}\ldots\tau_{b_{n}}\Lambda_{2g}\rangle\rangle_g &=\binom{g+\frac{n+1}{2}}{b_1,\ldots ,b_n}\mathbf{C}_{g,n+1} \\
&= \frac{\left(g+\frac{n-1}{2}\right)!(b_1+\ldots +b_n)}{b_1!\ldots b_n!}\mathbf{C}_{g,n}\frac{(2g+n+1)(2g+n)}{2\left(g+\frac{n+1}{2}\right)}\\
&=(2g+n)\sum_{i=1}^{n}\langle\langle \tau_{b_1}\ldots\tau_{b_{i-1}}\tau_{b_{i}-1}\tau_{b_{i+1}}\ldots\tau_{b_{n}}\Lambda_{2g}\rangle\rangle_g.
\end{align*}
For the dilaton equation: 
\begin{align*}
\langle\langle\tau_0\tau_1\tau_{b_1}\ldots\tau_{b_{n}}\Lambda_{2g}\rangle\rangle_g &= \binom{g+\frac{n+1}{2}}{0,1,b_1,\ldots,b_{2s+1}}\mathbf{C}_{g,n+1}\\
&=\binom{g+\frac{n-1}{2}}{b_1,\ldots,b_{n}}\mathbf{C}_{g,n}(2g+n)\left(g+\frac{n+1}{2}\right)\\
&=(2g+n)\left(g+\frac{n+1}{2}\right)\langle\langle\tau_{b_1}\ldots\tau_{b_{n}}\Lambda_{2g}\rangle\rangle_g.
\end{align*}
\end{proof}
Let us consider the following generating function:
\begin{equation}
    F := \sum_{n\geq 1}\frac{1}{n!}\sum_{b_1,\ldots,b_n\geq 0}\langle\langle\tau_{b_1}\ldots\tau_{b_n}\Lambda_{2g}\rangle\rangle_g\frac{t_{b_1}\ldots t_{b_n}}{(2g+n-2)!!}
\end{equation}
Then the string and dilaton equations can be written as follows:
\begin{align}
    \left(-\frac{\partial^2}{\partial t_0^2}+ \sum_{i=0}^\infty t_{i+1}\frac{\partial }{\partial t_i}\right)F:=L_{-1}F&=0\\
    \left(-\frac{\partial^2}{\partial t_0\partial t_1}+1+\sum_{i=0}^{\infty}it_i\frac{\partial }{\partial t_i}\right)F:=L_0F&=0 
\end{align}
It is easy to check that $\left[L_0,L_{-1}\right]=L_{-1}$. They look like two lowest Virasoro constraints. It would be interesting to investigate if we have higher Virasoro-like constraints as well. And of course, it would be of great interest to investigate string and dilaton equations for higher order integrals, i.e. for $\Lambda_{2k}$ with $k<g$.

\subsection{Explicit formulae for top degree terms}
We show how to compute the top degree terms $\langle\langle\tau_{b_1}\ldots\tau_{b_n}\rangle\rangle_g:= \langle\langle\tau_{b_1}\ldots\tau_{b_n}\Lambda_0\rangle\rangle_g$. These are sometimes called Witten terms because of the celebrated Witten's conjecture \cite{witten1991two}\footnote{By the way, do not confuse our double bracket notation with Witten's.}. First, we need to extract the highest degree term in $dH^g_{(d),\beta}/s!$. The result is:
\begin{equation}
    TT^g_n= \frac{d^s}{2^{s+2g}} \sum_{h=0}^g \frac{(2s-2h)!}{h!(s-h)!3^h}d^{2h} \sum_{\lambda\vdash (g-h)}\frac{\xi_{2\lambda}P_{2\lambda}}{|\text{Aut}\lambda|},
\end{equation}
where $P_\mu=P_\mu\left(\beta_1,\ldots,\beta_n\right)$ is the power sum, i.e $P_\mu=P_{\mu_1}P_{\mu_2}\ldots$ and $P_k=\beta_1^k+\ldots+\beta_n^k$. Then as usual, we compute the coefficient of $\beta_1^{b_1}\ldots\beta_n^{b_n}$ to obtain $\langle\langle\tau_{b_1}\ldots\tau_{b_n}\rangle\rangle_g$.

For a partition $\mu$, define the monomial symmetric function $m_\mu\left(x_1,x_2,\ldots\right):=\sum_{\gamma}x^\gamma$ where the sum ranges over all distinct permutations $\gamma=\left(\gamma_1,\gamma_2,\ldots\right)$ of the vector $\mu=\left(\mu_1,\mu_2,\ldots\right)$ and $x^\gamma= x_1^{\gamma_1}x_2^{\gamma_2}\ldots$. The functions $\left\{m_\mu,\mu\in\textbf{Part}\right\}$ form a basis of the algebra of symmetric functions. So for any partition $\lambda$, we have the expansion:
\begin{equation}
    P_\lambda=\sum_{\mu\vdash|\lambda|}R_{\lambda\mu}m_\mu.
\end{equation}
$R_{\lambda\mu}$ is equal to the number of ordered partitions $\pi=\left(A_1,\ldots,A_{l(\mu)}\right)$ of the set $\left\{1,\ldots,l(\lambda)\right\}$ such that for $1\leq j\leq l(\mu)$: 
\begin{equation*}
\mu_j=\sum_{i\in A_j}\lambda_i     .
\end{equation*}
For $2j\leq b_1+\ldots+b_n$, denote 
\begin{equation*}
  D_{2j}(\vec{b}):=\left\{\left(a_1,\ldots,a_n\right),a_i \text{ even}, a_i\leq b_i, a_1+\ldots+a_n=2j\right\}  .
\end{equation*}
 For a vector $\vec{a}$, denote $P_{\vec{a}}$ the associated partition, i.e. the rearrangement of the components of $\vec{a}$ in non-decreasing order. Using the obvious fact that $\left[x_1^{d_1}x_2^{d_2}\ldots\right]m_\mu(x)= 1$ if $\mu=P_{\vec{d}}$ and $0$ otherwise, we obtain:
\begin{prop}
For $b_1,\ldots,b_n\geq 0$, $b_1+\ldots+b_n=3g+\frac{n-1}{2}$, we have:
\begin{equation}
   \langle\langle\tau_{b_1}\ldots\tau_{b_n}\rangle\rangle_g=\frac{1}{2^{3g+\frac{n-1}{2}}}\sum_{h=0}^g  \frac{(2s-2h)!}{h!(s-h)!3^h} \sum_{\lambda\vdash (g-h)}\sum_{\vec{a}\in D_{2g-2h}(\vec{b})}\frac{\xi_{2\lambda}R_{2\lambda,P_{\vec{a}}}}{|\text{Aut}\lambda|}\binom{g+\frac{n-1}{2}+2h}{b_1-a_1,\ldots,b_n-a_n}.
\end{equation}
\end{prop}

In particular, for $g=1$: 
\begin{cor} For $n\geq 1$, $b_1,\ldots,b_n\geq 0$ and $b_1+\ldots+b_n=\frac{n+5}{2}$: 
\begin{equation}
    \langle\langle\tau_{b_1}\ldots\tau_{b_n}\rangle\rangle_1= \frac{(n+1)!}{3\times 2^{\frac{n+7}{2}}}\left[\frac{1}{n}\binom{\frac{n+5}{2}}{b_1,\ldots,b_n}+ \sum_{i=1}^n \binom{\frac{n+1}{2}}{b_1,\ldots,b_{i-1},b_i-2,b_{i+1},\ldots,b_n}\right],
\end{equation}
where $\binom{\frac{n+1}{2}}{b_1,\ldots,b_{i-1},b_i-2,b_{i+1},\ldots,b_n}=0$ if $b_i-2<0$.
\end{cor}
One should compare this formula with the Hodge integrals over $\overline{\mathcal{M}}_{1,n}$ which can be found, for instance, in \cite[Proposition 4.6.11]{lando2013graphs}:
\begin{prop}
For $d_1+\ldots d_n=n$, we have:
\begin{align}
    \langle \tau_{d_1}\ldots\tau_{d_n}\rangle:=\int_{\overline{\mathcal{M}}_{g,n}}\psi_1^{d_1}\ldots\psi_n^{d_n}= \frac{1}{24}\binom{n}{d_1\ldots d_n}\left(1-\sum_{i=2}^n\frac{(i-2)!(n-i)!}{n!}e_i(d_1,\ldots,d_n)\right),
\end{align}
where $e_i$ is the $i$-th elementary symmetric function:
\begin{align*}
    e_i(d_1,\ldots,d_n)=\sum_{j_1<\ldots<j_i}d_{j_1}\ldots d_{j_i}
\end{align*}
\end{prop}

\printbibliography

\end{document}